\documentclass[12pt,a4paper,leqno]{article}

\usepackage[nottoc]{tocbibind}
\usepackage{latexsym}
\usepackage{amssymb,exscale}
\usepackage[centertags]{amsmath}
\usepackage{amsthm}
\usepackage{footnpag}
\usepackage{verbatim} 

\usepackage[linktocpage]{hyperref}

\numberwithin{equation}{subsection}

\swapnumbers
\theoremstyle{definition}

\newtheorem{lemma}[equation]{Lemma}
\newtheorem{corollary}[equation]{Corollary}
\newtheorem{definition}[equation]{Definition}
\newtheorem{proposition}[equation]{Proposition}
\newtheorem{remark}[equation]{Remark}
\newtheorem*{theorem1}{Theorem 1}
\newtheorem*{theorem2}{Theorem 2}

\renewcommand{\phi}{\varphi}

\newcommand{\ti}{\tilde}

\renewcommand{\(}{\bigl(}
\renewcommand{\)}{\bigr)\vphantom{)}}

\newcommand{\ip}[2]{\langle#1,#2\rangle}

\newcommand{\imp}{$ \Longrightarrow $ }

\newcommand{\Cl}{\operatorname{Cl}}

\newcommand{\One}{{1\hskip-2.5pt{\rm l}}}

\newcommand{\si}{\sigma}
\newcommand{\ga}{\gamma}

\newcommand{\om}{\omega}
\newcommand{\Om}{\Omega}
\newcommand{\de}{\delta}

\newcommand{\al}{\alpha}
\newcommand{\be}{\beta}

\newcommand{\Ec}{\mathcal E}
\newcommand{\F}{\mathcal F}

\newcommand{\La}{\Lambda}

\newcommand{\Ex}{\mathbb E\,}
\newcommand{\R}{\mathbb R}

\renewcommand{\Pr}[1]{\mathbb{P}\mskip1.5mu\(\mskip1.5mu#1\mskip1.5mu\)}

\newcommand{\cE}[2]{\mathbb{E}\mskip1.5mu\(\mskip1.5mu#1\mskip1.5mu
 \big|\mskip1.5mu#2\mskip1.5mu\)}

\newcommand{\sif}{$\sigma$\nobreakdash-field}

\newcommand{\measurable}[1]{$#1$\nobreakdash-\hspace{0pt}measurable}
\newcommand{\dimensional}[1]{$#1$\nobreakdash-\hspace{0pt}dimensional}

\begin{document}

\title{Noise as a Boolean algebra of \sif s. I. Completion}

\author{Boris Tsirelson}

\date{}
\maketitle

\begin{abstract}
Nonclassical noises over the plane (such as the black noise of
percolation) consist of \sif s corresponding to some planar
domains. One can treat less regular domains as limits of more regular
domains, thus extending the noise and its set of \sif s. The greatest
extension is investigated in a new general framework.
\end{abstract}

\setcounter{tocdepth}{2}
\tableofcontents

\section*{Introduction}
\addcontentsline{toc}{section}{Introduction}
A noise is defined as a family of \sif s (in other words,
$\si$-algebras) contained in the \sif\ of a probability space and
satisfying some conditions. Initially, these \sif s were indexed by
intervals of the real line (the time axis). Recently, a spectacular
progress in understanding the full scaling limit of the critical
planar percolation \cite{SS} have lead to an important noise over the
plane (the black noise of percolation). Its \sif s are indexed by
planar domains with finite-length boundary; ``the regularity
assumption can be considerably weakened (though it cannot be
dropped)'' \cite[Remark 1.8]{SS}. The needed regularity of the domains
depends on some properties of the noise. For a classical (white,
Poisson or their combination) noise over $ \R^n $, the family of \sif
s extends naturally from nice domains to arbitrary Lebesgue measurable
subsets of $ \R^n $. What happens to a nonclassical (in particular,
black) noise? One may hope that it extends naturally to the greatest
class of subsets of $ \R^n $ acceptable for the given noise. For now,
nothing like that is proved, nor even conjectured.

It is worth to split the problem in two:

(a) enlarge the given set of \sif s (irrespective of their relation to
the domains in $ \R^n $);

(b) extend the given correspondence between the domains and the \sif
s.

Only the former problem, (a), is treated in this work.

A noise over $ \R $ extends readily from intervals to their finite
unions, which leads to a lattice homomorphism from the Boolean algebra
$ A $ of finite unions of intervals modulo finite sets to the lattice
$ \La $ of all sub-\sif s\footnote{%
 Each \sif\ is assumed to contain all null sets.}
of the \sif\ $ \F $ of a given probability space $ (\Om,\F,P) $:
\[
\F_{a\cap b} = \F_a \cap \F_b \, , \quad \F_{a\cup b}
= \F_a \vee \F_b \quad \text{for } a,b \in A \, ;
\]
here $ \F_a \vee \F_b $ is the least \sif\ containing both $ \F_a $
and $ \F_b $. The image $ B = \{ \F_a : a \in A \} $ is necessarily a
sublattice of $ \La $ and a Boolean algebra\footnote{%
 That is, Boolean lattice. I do not write ``Boolean sublattice''
 because the lattice of all \sif s is not Boolean.}
such that for all $ a,b \in A $,
\[
\text{if} \quad \F_a \cap \F_b
= \F_\emptyset \quad \text{then} \quad \F_a, \F_b \text{ are
independent}
\]
(that is, $ P(X\cap Y) = P(X) P(Y) $ for all $ X \in \F_a $, $
Y \in \F_b $). Thus $ B $ is an example to the following definition.

We define a \emph{noise-type Boolean algebra} (of \sif s) as a
sublattice $ B $ of $ \La $, containing the trivial \sif\ (only null
sets and their complements) and the whole $ \F $, such that $ B $ is a
Boolean algebra, and any two \sif s of $ B $ are independent whenever
their intersection is the trivial \sif.\footnote{%
 Homogeneity (that is, shift invariance) of a noise is ignored here.}

For the black noise of percolation we may start with the Boolean
algebra $ A_1 $ of finite unions of \dimensional{2} intervals $
(s_1,t_1) \times (s_2,t_2) \subset \R^2 $ modulo finite unions of
horizontal and vertical straight lines; or alternatively, the Boolean
algebra $ A_2 $ of all sets with finite-length boundary, modulo
finite-length sets; $ A_1 \subset A_2 $. We get two noise-type Boolean
algebras, $ B_1 \subset B_2 $. For every $ a \in A_2 $ there exist $
a_n \in A_1 $ such that $ a_n \uparrow a $ and therefore
$ \F_{a_n} \uparrow \F_a $; thus, the pair $ B_1, B_2 $ is an example
to the following definition.

Let $ B_1 \subset B_2 \subset \La $ be two noise-type Boolean
algebras. We say that $ B_1 $ is \emph{monotonically dense} in $ B_2 $
if every monotonically closed subset of $ \La $ containing $ B_1 $
contains also $ B_2 $. Here a subset $ Z \subset \La $ is called
monotonically closed if, first, $ Z $ contains $ \cap_n \Ec_n $ for
every decreasing sequence of \sif s $ \Ec_n \in Z $, and second, $ Z $
contains $ \vee_n \Ec_n $ (the least \sif\ containing all $ \Ec_n $)
for every increasing sequence of \sif s $ \Ec_n \in Z $.

We define the \emph{noise-type completion} of a noise-type Boolean
algebra $ B \subset \La $ as the greatest among all noise-type Boolean
algebras $ C \subset \La $ such that $ B \subset C $ and $ B $ is
monotonically dense in $ C $.

It appears that the greatest among these $ C $ exists and can be
described explicitly.

\begin{theorem1}
Every noise-type Boolean algebra has the noise-type completion.
\end{theorem1}

\begin{theorem2}
Let $ B \subset \La $ be a noise-type Boolean algebra. Denote by $ C $
its noise-type completion, and by $ \ti B $ the least monotonically
closed subset of $ \La $ containing $ B $. Then an arbitrary \sif\ $
\Ec \in \La $ belongs to $ C $ if and only if it satisfies the
following two conditions:

(a) $ \Ec \in \ti B $;

(b) $ \Ec $ has a complement $ \Ec' $ in $ \ti B $; that is,
$ \Ec' \in \ti B $, $ \Ec \cap \Ec' $ is the trivial \sif, and
$ \Ec \vee \Ec' $ is the whole $ \F $ (such $ \Ec' $ is necessarily
unique).
\end{theorem2}

Note that $ C $ is uniquely determined by $ \ti B $.

\section[The complete lattice of $\si$-fields]{\raggedright The complete
 lattice of \sif s}
\label{sec:1}
\subsection{Preliminaries: type $ L_2 $ subspaces}
\label{1a}

Let $ (\Om,\F,P) $ be a probability space, and $ H = L_2(\Om,\F,P) $
the corresponding Hilbert space, assumed to be separable. The
following two conditions on a (closed linear) subspace $ H_1 $ of $ H
$ are equivalent \cite[Th.~3]{Si}:

(a) $ H_1 $ is a sublattice of $ H $, containing constants. That is, $
H_1 $ contains $ f \vee g $ and $ f \wedge g $ for all $ f,g \in H_1
$, where $ (f \vee g)(\om) = \max\(f(\om),g(\om)\) $ and $ (f \wedge
g)(\om) = \min\(f(\om),g(\om)\) $; and $ H_1 $ contains the
one-dimensional space of constant functions.

(b) There exists a sub-\sif\ $ \F_1 \subset \F $ such that $ H_1 =
L_2(\F_1) $, the space of all \measurable{\F_1} functions of $ H $.

Such subspaces $ H_1 $ will be called type $ L_2 $ (sub)spaces. (In
\cite{Si} they are called measurable, which can be confusing.)

Each sub-\sif\ $ \F_1 \subset \F $ is assumed to contain all null
sets. Then, the relation $ H_1 = L_2(\F_1) $ establishes a bijective
correspondence between type $ L_2 $ subspaces of $ H $ and sub-\sif s of $
\F $. This correspondence is evidently isotone,
\[
H_1 \subset H_2 \quad \text{if and only if} \quad \F_1 \subset \F_2 \, .
\]
Thus, we may define a partially ordered set $ \La = \La(\Om,\F,P) $ in
two equivalent ways: as consisting of all type $ L_2 $ subspaces of $ H $,
or alternatively, of all sub-\sif s of $ \F $; up to isomorphism, it
is the same $ \La $. An element $ x \in \La $ may be thought of as a
type $ L_2 $ subspace $ H_x \subset H $ or a sub-\sif\ $ \F_x \subset \F $;
$ H_x = L_2(\F_x) $.

The set $ \La $ contains the greatest element $ 1 $ (the whole $ H $, or
the whole $ \F $) and the least element $ 0 $ (the one-dimensional
space of constants, or the trivial \sif, --- only null sets and their
complements).

The infimum exists for every subset of $ \La $, since the intersection
of type $ L_2 $ spaces is a type $ L_2 $ space; alternatively, the intersection
of \sif s is a \sif.

Existence of the supremum follows readily \cite[Th.~2.31]{DP}. It is
the type $ L_2 $ space generated by the union of the given type $ L_2 $
spaces. Alternatively, it is the \sif\ generated by the union of the
given \sif s. (See \cite[Th.~2]{Si}.)

Thus, $ \La $ is a complete lattice. For two elements $ x,y \in \La $
their infimum is denoted by $ x \wedge y $, and supremum by $ x \vee y
$.

\subsection{Bad properties}
\label{1b}

This subsection is not used in the sequel and may be skipped. Its goal
is, to warn the reader against some incorrect arguments that could
suggest themselves.

See \cite[Sect.~4]{DP} about modular and distributive lattices, the
diamond $ M_3 $ and the pentagon $ N_5 $.

\begin{remark}\label{1b2}
The lattice $ \La $ is not modular, and therefore not distributive,
unless it is finite.

If $ (\Om,\F,P) $ consists of only a finite number $ n $ of atoms (and
no nonatomic part) then $ \dim H = n $. For $ n=3 $, $ \La=M_3 $ is the
diamond, modular but not distributive. For $ n=4 $, $ \La $ is not
modular, since it  contains $ N_5 $. Proof: let $ \al,\be,\ga,\de $ be
the four atoms of $ (\Om,\F,P) $, then $ \{ 0,1,u,v,w \} = N_5 $ where
$ u = \si(\al,\ga,\be\cup\de) $, $ v = \si(\al\cup\be,\ga\cup\de) $
and $ w = \si(\al\cup\ga,\be\cup\de) $; here $ \si(\dots) $ is the
\sif\ generated by $ (\dots) $.
\end{remark}

The following two remarks show that the lattice operations, $ (x,y)
\mapsto x \wedge y $ and $ (x,y) \mapsto x \vee y $, generally violate
some natural continuity.

\begin{remark}
It may happen that $ x_n \uparrow x $ (that is, $ x_1 \le x_2 \le
\dots $ and $ \sup_n x_n = x $), $ x_n \wedge y = 0 $ for all $ n $,
but $ x \wedge y = y \ne 0 $.

Proof. Assuming that $ (\Om,\F,P) $ contains $ \al_1, \al_2, \dots \in
\F $ such that $ \al_1 \subset \al_2 \subset \dots $, $ P(\al_1) <
P(\al_2) < \dots $ and $ \lim_n P(\al_n) < 1 $, we introduce $ x_n =
\si(\al_1,\dots,\al_n) $, $ x = \si(\al_1,\al_2,\dots) $ and $ y =
\si(\al) $ where $ \al = \cup_n \al_n $. Then $ x_n \uparrow x $; $
x_n \wedge y = 0 $ (just because $ \al \notin \si(\al_1,\dots,\al_n)
$); and $ x \wedge y = y \ne 0 $ (since $ \al \in
\si(\al_1,\al_2,\dots) $ and $ 0 < P(\al) < 1 $).
\end{remark}

\begin{remark}
It may happen that $ x_n \downarrow 0 $, $ x_n \vee y = 1 $ for all $
n $, but $ y \ne 1 $.

Proof. Assuming that $ (\Om,\F,P) $ is the interval $ [0,1) \subset \R
$ with Lebesgue measure, we introduce the \sif\ $ y $ of all measurable
sets invariant under the transformation $ \om \mapsto 1-\om $, and
(for each $ n $) the \sif\ $ x_n $ of all measurable sets invariant
under the transformation $ \om \mapsto \om+2^{-n} \pmod 1 $. The rest
is left to the reader.
\end{remark}

See also Remark \ref{2a65}.

\subsection{More preliminaries: tensor products}
\label{1c}

The product of two probability spaces leads to the tensor product of
Hilbert spaces (assumed to be separable, as before), see for instance
\cite[Sect. II.4, Th.~II.10]{RS}; that is,
\[
L_2 \( (\Om',\F',P') \times (\Om'',\F'',P'') \) = L_2 (\Om',\F',P')
\otimes L_2 (\Om'',\F'',P'')
\]
in the following sense: the formula $ ( f \otimes g ) (\om',\om'') =
f(\om') g(\om'') $ establishes a unitary operator between these two
Hilbert spaces.

The same situation appears when two sub-\sif s $ \F_1, \F_2 \subset \F
$ are independent. The formula
\[
( f \otimes g ) (\om) = f(\om) g(\om) \quad \text{for } f \in
L_2(\F_1), \, g \in L_2(\F_2), \, \om \in \Om
\]
establishes a unitary operator from $ L_2(\F_1) \otimes L_2(\F_2) $
onto $ L_2 ( \F_1 \vee \F_2 ) $. This operator is the composition of
the operator $ L_2(\F_1) \otimes L_2(\F_2) \to L_2 \( (\Om,\F_1,P)
\times (\Om,\F_2,P) \) $ discussed before, and the operator $ L_2 \(
(\Om,\F_1,P) \times (\Om,\F_2,P) \) \to L_2 ( \Om, \F_1 \vee \F_2, P )
$ conjugated to the measure preserving ``diagonal'' map
\[
( \Om, \F_1 \vee \F_2, P ) \ni \om \mapsto (\om,\om) \in (\Om,\F_1,P)
\times (\Om,\F_2,P) \, .
\]

We need also the equality
\begin{equation}\label{1c2}
( H_{1\mathrm a} \otimes H_{2\mathrm a} ) \cap ( H_{1\mathrm b}
\otimes H_{2\mathrm b} ) = ( H_{1\mathrm a} \cap H_{1\mathrm b} )
\otimes ( H_{2\mathrm a} \cap H_{2\mathrm b} )
\end{equation}
for arbitrary (closed linear) subspaces $ H_{1\mathrm a}, H_{1\mathrm
b} \subset H_1 $ and $ H_{2\mathrm a}, H_{2\mathrm b} \subset H_2 $ of
Hilbert spaces $ H_1, H_2 $. Formula \eqref{1c2} follows easily from
its special case
\[
( H_{1\mathrm a} \otimes H_2 ) \cap ( H_1 \otimes H_{2\mathrm a} ) =
H_{1\mathrm a} \otimes H_{2\mathrm a}
\]
for $ H_{1\mathrm a} \subset H_1 $, $ H_{2\mathrm a} \subset H_2
$. This special case follows from the evident equality
\[
( Q_1 \otimes
\One ) ( \One \otimes Q_2 ) = Q_1 \otimes Q_2 = ( \One \otimes Q_2 ) (
Q_1 \otimes \One ) \, ,
\]
since for two commuting projections, the image of their product is the
intersection of their images.

\subsection{Good properties}
\label{1d}

Elements $ x \in \La $ may be treated as sub-\sif s $ \F_x \subset \F
$ or type $ L_2 $ subspaces $ H_x \subset H = L_2(\Om,\F,P) $, but also
as the corresponding orthogonal projections $ Q_x : H \to H $, $ Q_x H
= H_x $, which gives us some useful structures on $ \La $ not
derivable from the partial order.

The strong operator topology on the projection operators $ Q_x $ gives
us a topology on $ \La $; we call it the strong operator topology on $
\La $. It is metrizable (since the strong operator
topology is metrizable on operators of norm $ \le 1 $). Thus,
\[
x_n \to x \quad \text{means} \quad \forall \psi \in H \;\; \|
Q_{x_n} \psi - Q_x \psi \| \to 0 \, .
\]
Below, ``topologically'' means ``according to the strong operator
topology''.

On the other hand we have the monotone convergence derived from the
partial order on $ \La $:
\begin{gather*}
x_n \downarrow x \quad \text{means} \quad x_1 \ge x_2 \ge \dots \text{
  and } \inf_n x_n = x \, , \\
x_n \uparrow x \quad \text{means} \quad x_1 \le x_2 \le \dots \text{
  and } \sup_n x_n = x \, .
\end{gather*}

\begin{definition}
(a) A set $ Z \subset \La $ is \emph{monotonically closed,} if for all
$ x_n \in Z $ and $ x \in \La $
\begin{gather*}
x_n \downarrow x \quad \text{implies} \quad x \in Z \, , \\
x_n \uparrow x \quad \text{implies} \quad x \in Z \, .
\end{gather*}
(b) Given two subsets $ Z_1 \subset Z_2 \subset \La $, we say that $
Z_1 $ is \emph{monotonically dense} in $ Z_2 $ if every monotonically
closed set containing $ Z_1 $ contains also $ Z_2 $.
\end{definition}

\begin{lemma}\label{1d2}
(a) $ x_n \downarrow x $ implies $ x_n \to x $; also, $ x_n \uparrow x
$ implies $ x_n \to x $;

(b) every set closed in the strong operator topology is monotonically
closed.
\end{lemma}

\begin{proof}
Clearly, (b) follows from (a); we have to prove (a). 

First, if $ x_n \downarrow x $ then $ H_{x_1} \supset H_{x_2} \supset
\dots $ and $ \cap_n H_{x_n} = H_x $, therefore $ Q_{x_n} \to Q_x $
strongly. Second, if $ x_n \uparrow x $ then $ H_{x_1} \subset H_{x_2}
\subset \dots $; the closure of $ \cup_n H_{x_n} $, being a type $ L_2
$ space, is equal to $ H_x $; therefore $ Q_{x_n} \to Q_x $ strongly.
\end{proof}

\begin{definition}\label{1d3}
Elements $ x,y \in \La $ are \emph{commuting,}\footnote{%
 Not to be confused with the notion mentioned in \cite[Chap.~II,
 Sect.~14, p.~52]{Bi}.}
if $ Q_x Q_y = Q_y Q_x $. A subset of $ \La $ is \emph{commutative,}
if its elements are pairwise commuting.
\end{definition}

Note that
\begin{gather}
\text{the topological closure of a commutative set is
 commutative,} \label{1d33} \\
\text{if } x,y \in \La \text{ are commuting then } Q_x Q_y =
Q_{x\wedge y} \, . \label{1d37}
\end{gather}

\begin{proposition}\label{1d4a}
Let
\[
B \subset C \subset \La \, , \quad B \text{ is commutative,} \quad
\forall x,y \in B \;\; x \wedge y \in B \, .
\]
Then $ B $ is monotonically dense in $ C $ if and only if $ B $ is
topologically dense in $ C $ (that is, $ C $ is contained in the
topological closure of $ B $).
\end{proposition}

The ``only if'' part follows from \ref{1d2}(b). The proof of the
``if'' part is given after a lemma.

Given $ x_n \in \La $, we define
\[
\liminf_n x_n = \sup_n \inf_k x_{n+k} \, , \quad
\limsup_n x_n = \inf_n \sup_k x_{n+k} \, .
\]

\begin{lemma}\label{1d5a}
If $ x_n \in \La $ are pairwise commuting and $ x_n \to x $ then
\[
\liminf_k x_{n_k} = x
\]
for some $ n_1 < n_2 < \dots $
\end{lemma}

\begin{proof}
The commuting projection operators $ Q_{x_n} $ generate a commutative
von Neumann algebra; such algebra is always isomorphic to the algebra
$ L_\infty $ on some measure space (of finite measure), see for
instance \cite[Th.~1.22]{Ta}. Denoting the isomorphism by $ \al $ we
have $ \al(Q_{x_n}) = \One_{E_n} $, $ \al(Q_x) = \One_E $ (indicators
of measurable sets $ E_n, E $). By \eqref{1d37},
\[
\al ( Q_{x_m \wedge x_n} ) = \One_{E_m \cap E_n}
\]
for all $ m,n $; the same holds for more than two indices.

The strong convergence of operators $ Q_{x_n} \to Q_x $ implies
convergence in measure of indicators, $ \One_{E_n} \to \One_E $ (since
$ \One_{E_n} = Q_{x_n} \One \to Q_x \One = \One_E $). We choose a
subsequence convergent almost everywhere, $ \One_{E_{n_k}} \to \One_E
$, then $ \liminf_k \One_{E_{n_k}} = \One_E $, that is,
\[
\sup_k \inf_i \One_{E_{n_{k+i}}} = \One_E \, .
\]
We have $ \al ( Q_{x_{n_k} \wedge x_{n_{k+1}} \wedge \dots \wedge
x_{n_{k+i}}} ) = \One_{E_{n_k} \cap E_{n_{k+1}} \cap \dots \cap
E_{n_{k+i}}} $, therefore (for $ i \to \infty $, using \ref{1d2}(a)), $
\al ( Q_{\inf_i x_{n_{k+i}}} ) = \inf_i \One_{E_{n_{k+i}}} $, and
further (for $ k \to \infty $), $ \al ( Q_{ \sup_k \inf_i x_{n_{k+i}}}
) = \sup_k \inf_i \One_{E_{n_{k+i}}} $. We get $ \al ( Q_{\liminf_k
x_{n_k}} ) = \liminf_k \One_{E_{n_k}} = \One_E = \al(Q_x) $, therefore
$ \liminf_k x_{n_k} = x $.
\end{proof}

\begin{proof}[Proof of Proposition \ref{1d4a}, the ``if'' part]
Let $ Z $ be a monotonically closed set, $ Z \supset B $, and $ x \in
C $; we have to prove that $ x \in Z $.

There exist $ x_n \in B $ such that $ x_n \to x $. By \ref{1d5a} we may
assume that $ \liminf_n x_n = x $. We have $ x_n \wedge x_{n+1} \wedge
\dots \wedge x_{n+k} \in B \subset Z $ for all $ k $ and $ n $, which
implies $ \inf_k x_{n+k} \in Z $ and further $ x = \sup_n \inf_k
x_{n+k} \in Z $.
\end{proof}

It will be shown (see \ref{2a2}) that every noise-type Boolean algebra
$ B $ is a commutative subset of $ \La $. Thus, by Proposition
\ref{1d4a}, the following definition is equivalent to that of the
introduction.

\begin{definition}\label{1d6}
The \emph{noise-type completion} of a noise-type Boolean algebra $ B
\subset \La $ is the greatest among all noise-type Boolean algebras $
C \subset \La $ such that $ B \subset C $ and $ B $ is topologically
dense in $ C $ (according to the strong operator topology).
\end{definition}

Also, by \ref{1d4a}, the least monotonically closed subset of $ \La $
containing $ B $ is equal to the topological closure of $ B $.

\begin{corollary}\label{1d7}
In Theorem 2, the set $ \ti B $ may be replaced with the topological
closure of $ B $ (according to the strong operator topology).
\end{corollary}

Thus, the ``monotonical'' notions are eliminated.

\smallskip

\textsc{From now on, convergence, denseness and closeness are always
topological (according to the strong operator topology). No other
topology on $ \La $ will be used.}

\smallskip

\begin{proposition}\label{1d4}
Let $ x_n, y_n, x, y \in \La $, $ x_n \to x $, $ y_n \to y $, and for
each $ n $ (separately), $ x_n, y_n $ commute. Then $ x_n \wedge y_n
\to x \wedge y $.
\end{proposition}

\begin{proof}
Note that $ Q_{x_n} Q_{y_n} \to Q_x Q_y $, since for each $ \psi \in H
$,
\[
\| Q_{x_n} Q_{y_n} \psi - Q_x Q_y \psi \| \le \| Q_{x_n} \| \cdot \|
(Q_{y_n}-Q_y) \psi \| + \| (Q_{x_n}-Q_x) Q_y \psi \| \to 0 \, .
\]
Similarly, $ Q_{y_n} Q_{x_n} \to Q_y Q_x $. We have $ Q_{x_n} Q_{y_n}
= Q_{y_n} Q_{x_n} $, therefore $ Q_x Q_y = Q_y Q_x $. By \eqref{1d37},
$ Q_{x\wedge y} = Q_x Q_y $. Similarly, $ Q_{x_n \wedge y_n} = Q_{x_n}
Q_{y_n} $. We get $ Q_{x_n \wedge y_n} \to Q_{x\wedge y} $, that is, $
x_n \wedge y_n \to x \wedge y $.
\end{proof}

\begin{definition}\label{1d9}
Elements $ x,y \in \La $ are \emph{independent,} if the
corresponding \sif s $ \F_x $, $ \F_y $ are independent.
\end{definition}

It means, $ P(X\cap Y) = P(X) P(Y) $ for all $ X \in \F_x $, $ Y \in
\F_y $. Or equivalently, $ \ip{ Q_x \xi }{ Q_y \psi } = \ip{ Q_x \xi
}{ \One } \ip{ Q_y \psi }{ \One } $ for all $ \xi, \psi \in H $.

\begin{proposition}\label{1d45}
The following two conditions on $ x,y \in \La $ are equivalent:

(a) $ x,y $ are independent;

(b) $ x,y $ commute, and $ x \wedge y = 0 $.
\end{proposition}

\begin{proof}
(a) \imp (b): independence of $ \F_x, \F_y $ implies $ \cE{ f }{ \F_y
} = \Ex f $ for all $ f \in L_2(\F_x) $; that is, $ Q_y f = \ip{ f }{
\One } \One $ for $ f \in H_x $, and therefore $ Q_y Q_x = Q_0 = Q_x
Q_y $.

(b) \imp (a): by \eqref{1d37}, $ Q_y Q_x = Q_0 = Q_x Q_y $; thus $ Q_y
f = \ip{ f }{ \One } \One $ for $ f \in H_x $, and therefore $ \Pr{ A
\cap B } = \ip{ \One_A }{ \One_B } = \ip{ \One_A }{ Q_y \One_B } = \ip{
Q_y \One_A }{ \One_B } = \ip{ \One_A }{ \One } \ip{ \One }{ \One_B }
= \Pr A \Pr B $ for all $ A \in \F_x $, $ B \in \F_y $.
\end{proof}

It follows that all pairs $ (x,y) \in \La \times \La $ such that $ x,y
$ are independent are a closed set in $ \La \times \La $ (in the
product topology).

It may happen that $ x \wedge y = 0 $ but $ x,y $ do not
commute.\footnote{%
 For example, $ v $ and $ w $ of \ref{1b2} are independent if and only
 if $ P(\al) P(\de) = P(\be) P(\ga) $; $ u $ and $ w $ are never
 independent.}

For every $ x \in \La $ the triple $ (\Om,\F_x,P|_{\F_x}) $ is also a
probability space, and it may be used similarly to $ (\Om,\F,P) $,
giving the complete lattice $ \La(\F_x) = \La(\Om,\F_x,P|_{\F_x}) $
endowed with the topology, etc. The evident lattice isomorphism
\[
\La_x = \{ y \in \La : y \le x \} \cong \La(\F_x)
\]
is also a homeomorphism. Proof: if $ y \le x $ then $ H_y \subset H_x
\subset H $ and therefore $ Q_y = Q_y^{(x)} Q_x $ where $ Q_y^{(x)} :
H_x \to H_x $ is the orthogonal projection onto $ H_y $. It follows
that $ Q_{y_n} \to Q_y $ if and only if $ Q_{y_n}^{(x)} \to Q_y^{(x)}
$. That is, $ y_n \to y $ in $ \La_x $ if and only if $ y_n \to y $ in
$ \La(\Om,\F_x,P) $.

Given $ x,y \in \La $, the product set $ \La_x \times \La_y $
carries the product topology and the product partial order, and is
again a lattice (see \cite[Sect.~2.15]{DP} for the product of two
lattices), moreover, a complete lattice (see \cite[Exercise
2.26(ii)]{DP}).

\begin{proposition}\label{1d5}
If $ x,y \in \La $ are independent then the map
\[
\La_x \times \La_y \ni (u,v) \mapsto u \vee v \in \La_{x\vee y}
\]
is an embedding, both algebraically and topologically. In other words,
this map is both a lattice isomorphism and a homeomorphism between $
\La_x \times \La_y $ and its image $ \La_{x,y} = \{ u \vee v : u \in
\La_x, v \in \La_y \} $ treated as a sublattice and a topological
subspace of $ \La_{x\vee y} $.
\end{proposition}

\begin{remark}
If $ x \wedge y = 0 $ but $ x,y $ are not independent then the map
need not be one-to-one. For example, \ref{1b2} gives us $ u,v,w $ such
that $ u \wedge v = 0 $, $ w<u $ and $ w \vee v = 1 $. Thus, the map $
\La_u \times \La_v \to \La $ sends to $ 1 $ both $ (u,v) $ and $ (w,v)
$.
\end{remark}

\begin{proof}[Proof of Proposition \ref{1d5}]
According to Sect.~\ref{1c}, independence of $ x,y $ implies
\[
H_{x\vee y} = H_x \otimes H_y \, , \quad (f\otimes g)(\cdot) =
f(\cdot) g(\cdot) \quad \text{for } f \in H_x, \, g \in H_y \, .
\]
We may treat $ \La_x $ as consisting of all $ L_2 $-type subspaces $
H_u \subset H_x $, or the corresponding projections $ Q_u : H_x \to
H_x $. The same holds for $ \La_y $ and $ \La_{x\vee y}
$.

Treating $ H_x \otimes H_y $ as $ H_{x\vee y} $ we get $ H_u \otimes
H_v = H_{u\vee v} \subset H_{x\vee y} $, thus, $ Q_u \otimes Q_v =
Q_{u\vee v} $ for $ u \in \La_x $, $ v \in \La_y $. If $ u_n \to u $
and $ v_n \to v $ then $ u_n \vee v_n \to u \vee v $, since $ Q_{u_n
\vee v_n} = Q_{u_n} \otimes Q_{v_n} \to Q_u \otimes Q_v $. It means
that the map $ J : \La_x \times \La_y \to \La_{x \vee y} $, $
J(u,v)=u\vee v $, is continuous. This map preserves lattice
operations, that is,
\begin{equation}\label{1d11}
\begin{gathered}
( u_1 \vee v_1 ) \vee ( u_2 \vee v_2 ) = ( u_1 \vee u_2 ) \vee ( v_1
  \vee v_2 ) \, , \\
( u_1 \vee v_1 ) \wedge ( u_2 \vee v_2 ) = ( u_1 \wedge u_2 ) \vee (
  v_1 \wedge v_2 )
\end{gathered}
\end{equation}
for all $ u_1, u_2 \in \La_x $, $ v_1, v_2 \in \La_y $. The former
equality is trivial; the latter equality follows from \eqref{1c2}
applied to $ H_1 = H_x, H_2 = H_y, H_{1\mathrm a} = H_{u_1},
H_{2\mathrm a} = H_{v_1}, H_{1\mathrm b} = H_{u_2}, H_{2\mathrm b} =
H_{v_2} $. We have to prove that $ J $ is one-to-one and the
inverse map is continuous.

By \eqref{1d11}, $ ( u \vee v ) \wedge x = u $ and $ ( u \vee v )
\wedge y = v $ for all $ u \in \La_x $, $ v \in \La_y $. Thus $ J $ is
one-to-one. It remains to prove that the maps $ z \mapsto z \wedge x
$, $ z \mapsto z \wedge y $ are continuous on $ J(\La_x \times \La_y)
$. Let $ z_n, z \in J(\La_x \times \La_y) $, $ z_n \to z $. We
introduce $ u_n = z_n \wedge x $, $ u = z \wedge x $, $ v_n = z_n
\wedge y $, $ v = z \wedge y $, then $ u_n, u \in \La_x $, $ v_n, v
\in \La_y $, $ u_n \vee v_n = z_n $ and $ u \vee v = z $. We note that
$ ( Q_{u_n} \otimes Q_{v_n} ) ( Q_x \otimes Q_0 ) = Q_{u_n} \otimes
Q_0 = ( Q_x \otimes Q_0 ) ( Q_{u_n} \otimes Q_{v_n} ) $, that is, $
z_n, x $ commute (for each $ n $ separately). By \ref{1d4}, $ z_n
\wedge x \to z \wedge x $, that is, $ u_n \to u $. Similarly, $ v_n
\to v $.
\end{proof}

\begin{remark}\label{1d14}
We see that $ \La_x \times \La_y $ (for independent $ x,y \in \La $)
is naturally isomorphic to the sublattice
\[
\La_{x,y} = \{ u \vee v : u \in \La_x, v \in \La_y \} = \{ u
\vee v : u \le x, v \le y \} \subset \La \, .
\]
The correspondence between a pair $ (u,v) \in \La_x \times \La_y $ and
$ z \in \La_{x,y} $ is given by
\begin{gather*}
z = u \vee v \, , \\
u = z \wedge x \, , \quad v = z \wedge y \, .
\end{gather*}
Therefore
\begin{equation}\label{1d12}
\La_{x,y} = \{ z \in \La : z = ( z \wedge x ) \vee ( z \wedge y ) \}
\, .
\end{equation}
The continuous map
\[
\La_{x,y} \ni z \mapsto z \wedge x \in \La_x
\]
is a lattice homomorphism (and the same holds for the similar map with
$ y $ in place of $ x $):
\begin{equation}\label{1d16}
\forall z_1,z_2 \in \La_{x,y} \quad ( z_1 \vee z_2 ) \wedge x = ( z_1
\wedge x ) \vee ( z_2 \wedge x )
\end{equation}
and of course, $ ( z_1 \wedge z_2 ) \wedge x = ( z_1 \wedge x )
\wedge ( z_2 \wedge x ) $. Thus, any relation between elements of $
\La_{x,y} $ expressed in terms of lattice operations is equivalent to
the conjunction of two similar relations ``restricted'' to $ x $ and $
y $. For example, the relation
\[
( z_1 \vee z_2 ) \wedge z_3 = z_4 \vee z_5
\]
between $ z_1, z_2, z_3, z_4, z_5 \in \La_{x,y} $ splits in two:
\[
( ( z_1 \wedge x ) \vee ( z_2 \wedge x ) ) \wedge ( z_3 \wedge x ) = (
z_4 \wedge x ) \vee ( z_5 \wedge x )
\]
and a similar relation with $ y $ in place of $ x $.
\end{remark}

\section[Noise-type Boolean algebras]{\raggedright Noise-type Boolean
 algebras}

\label{sec:2}
\subsection{Distributivity relations}
\label{2a}

Throughout this section $ B \subset \La $ is a noise-type Boolean
algebra, as defined below.

\begin{definition}
A noise-type Boolean algebra is a sublattice $ B $ of $ \La $,
containing $ 0 $ and $ 1 $, such that $ B $ is a Boolean algebra and
all $ x,y \in B $ satisfying $ x \wedge y = 0 $ are independent.
\end{definition}

Every $ x \in B $ has its complement $ x' \in B $;
\[
x \wedge x' = 0 \, , \quad x \vee x' = 1 \, .
\]
(The complement in $ B $ is unique, however, many other complements
may exist in $ \La $.)

\begin{lemma}\label{2a2}\footnote{%
 Well-known long ago (in slightly different form).}
$ Q_x Q_y = Q_{x\wedge y} $ for all $ x,y \in B $.
\end{lemma}

\begin{proof}
For every $ \psi \in H $ of the form $ \psi = \psi_{00} \psi_{01}
\psi_{10} \psi_{11} $ where $ \psi_{11} \in H_{x\wedge y} $, $
\psi_{10} \in H_{x\wedge y'} $, $ \psi_{01} \in H_{x'\wedge y} $, $
\psi_{00} \in H_{x'\wedge y'} $, we have
\begin{gather*}
Q_y \psi = Q_y ( (\psi_{00}\psi_{10}) (\psi_{01}\psi_{11}) ) =
 \ip{ \psi_{00}\psi_{10} }{ \One } \psi_{01}\psi_{11} =
 \ip{ \psi_{00} }{ \One } \ip{ \psi_{10} }{ \One } \psi_{01}\psi_{11}
 \, ; \\
Q_x Q_y \psi = \ip{ \psi_{00} }{ \One } \ip{ \psi_{10} }{ \One } \ip{
 \psi_{01} }{ \One } \psi_{11} \, ; \\
Q_{x\wedge y} \psi = Q_{x\wedge y} ( ( \psi_{00}\psi_{01}\psi_{10} )
 \psi_{11} ) = \ip{ \psi_{00}\psi_{01}\psi_{10} }{ \One } \psi_{11} =
 Q_x Q_y \psi \, .
\end{gather*}
It follows that $ Q_x Q_y = Q_{x\wedge y} $, since linear combinations
of the considered $ \psi $ are dense in $ H $.
\end{proof}

We denote by $ \Cl(B) $ the closure\footnote{%
 Topological, of course; recall the note after \ref{1d7}.}
of $ B $;
\begin{equation}\label{2a3}
\Cl(B) \text{ is commutative,}\footnotemark
\end{equation}
\footnotetext{As defined by \ref{1d3}.}
and for all $ x,y \in \Cl(B) $,
\begin{gather}
x \wedge y \in \Cl(B) \, , \label{2a4a} \\
x \wedge y = 0 \text{ if and only if } x,y \text{ are
 independent.}\footnotemark \label{2a5a}
\end{gather}
\footnotetext{As defined by \ref{1d9}.}
\emph{Proof:} by \eqref{2a2}, $ B $ is commutative, which implies
\eqref{2a3} by \eqref{1d33}; \eqref{2a4a} follows from \ref{1d4} and
\eqref{2a3}; \eqref{2a5a} follows from \ref{1d45} and \eqref{2a3}.

By \ref{1d4} and \eqref{2a3},
\begin{equation}\label{2a6a}
x_n \wedge y_n \to x \wedge y
\end{equation}
whenever $ x_n,x,y_n,y \in \Cl(B) $, $ x_n \to x $, $ y_n \to y $.

\begin{remark}\label{2a65}
Surprisingly, $ x_n \vee y_n $ need not converge to $ x \vee y $, even
if $ x_n \in B $, $ x_n \downarrow 0 $, $ y_n = x'_n $; it may happen
that $ y_n \uparrow y $, $ y \ne 1 $. ``The phenomenon \dots tripped
up even Kolmogorov and Wiener'' \cite[p.~48]{Wi}. Also, it may happen
that $ x_n \in B $, $ x_n \to 0 $, $ y_n = x'_n $, and the projectors
$ Q_{y_n} $ converge weakly (that is, in the weak operator topology)
to an operator that is not a projection, and therefore no subsequence
of $ (x'_n)_n $ converges in $ \La $.

On the other hand it can be shown that if $ x_n \in B $, $ x_n \to 1
$, then necessarily $ x'_n \to 0 $.
\end{remark}

\begin{lemma}
\[
\forall x \in \Cl(B) \;\; \forall y,z \in B \quad x \wedge ( y \vee z
) = ( x \wedge y ) \vee ( x \wedge z ) \, .
\]
\end{lemma}

\begin{proof}
First, consider the case $ y \wedge z = 0 $.
By \eqref{2a5a}, $ y, z $ are independent.
We take $ x_n \in B $ such that $ x_n \to x $.
By \eqref{2a6a}, $ x_n \wedge y \to x \wedge y $, $ x_n \wedge z \to x
\wedge z $ and $ x_n \wedge ( y \vee z ) \to x \wedge ( y \vee z ) $.
Applying \ref{1d5} to $ ( x_n \wedge y, x_n \wedge z ) \in \La_y
\times \La_z $ we get $ ( x_n \wedge y ) \vee ( x_n \wedge z ) \to (
x \wedge y ) \vee ( x \wedge z ) $. On the other hand, $ ( x_n \wedge
y ) \vee ( x_n \wedge z ) = x_n \wedge ( y \vee z ) $ (since $ B $ is
distributive). Thus, $ x_n \wedge ( y \vee z ) \to ( x \wedge y ) \vee
( x \wedge z ) $ and so, $ x \wedge ( y \vee z ) = ( x \wedge y ) \vee
( x \wedge z ) $.

Second, if $ y \wedge z \ne 0 $, we introduce $ u = y \wedge z' $, $ v
= y \wedge z $, $ w = y' \wedge z $, note that $ u \vee v \vee w = y
\vee z $, $ u \wedge v = 0 $, $ u \wedge w = 0 $, $ v \wedge w = 0 $,
and apply several times the special case proved above:
\begin{gather*}
x \wedge y = x \wedge ( u \vee v ) = ( x \wedge u ) \vee ( x \wedge v
 ) \, ; \\
x \wedge z = x \wedge ( v \vee w ) = ( x \wedge v ) \vee ( x \wedge w
 ) \, ;
\end{gather*}
\vspace{-9mm}
\begin{multline*}
x \wedge ( y \vee z ) = x \wedge ( u \vee ( v \vee w ) ) = ( x \wedge
 u ) \vee ( x \wedge ( v \vee w ) ) = \\
= ( x \wedge u ) \vee ( x \wedge v ) \vee ( x \wedge w ) \, .
\end{multline*}
\end{proof}

Taking $ y = z' $ we get $ x = ( x \wedge z ) \vee ( x \wedge z' ) $,
that is (recall \eqref{1d12}),
\[
\forall z \in B \quad \Cl(B) \subset \La_{z,z'} \, .
\]
By \eqref{1d16},
\begin{equation}\label{2a4}
\forall x,y \in \Cl(B) \;\; \forall z \in B \quad ( x \vee y ) \wedge
z = ( x \wedge z ) \vee ( y \wedge z ) \, .
\end{equation}

\begin{proposition}\label{2a5}
Let $ x,y \in \Cl(B) $, $ x \wedge y = 0 $, $ x \vee y = 1 $. Then $
\Cl(B) \subset \La_{x,y} $, that is,
\[
\forall z \in \Cl(B) \quad z = ( x \wedge z ) \vee ( y \wedge z ) \,
.
\]
\end{proposition}

\begin{proof}
We take $ z_n \in B $ such that $ z_n
\to z $ and apply \eqref{2a4}:
\[
\forall n \quad z_n = ( x \wedge z_n ) \vee ( y \wedge z_n ) \, .
\]
By \eqref{2a5a}, $ x $ and $ y $ are independent. We have $ z_n = u_n
\vee v_n \in \La_{x,y} $ where
\[
u_n = x \wedge z_n \in \La_x \, , \quad v_n = y \wedge z_n \in \La_y
\, .
\]
By \eqref{2a6a}, $ u_n \to u = x \wedge z $ and $ v_n \to v = y \wedge
z $. By \ref{1d5}, $ u_n \vee v_n \to u \vee v $. On the
other hand, $ u_n \vee v_n = z_n \to z $. Thus, $ z = u \vee v = ( x
\wedge z ) \vee ( y \wedge z ) $.
\end{proof}

\begin{lemma}\label{2a6}
For every $ x \in \Cl(B) $ there exists at most one $ y \in \Cl(B) $
such that $ x \wedge y = 0 $ and $ x \vee y = 1 $.
\end{lemma}

\begin{proof}
Assume that $ y_1,y_2 \in \Cl(B) $, $ x \wedge y_k = 0 $ and $ x \vee
y_k = 1 $ for $ k=1,2 $. By \ref{2a5}, $ y_2 \in \La_{x,y_1} $, that
is, $ y_2 = ( x \wedge y_2 ) \vee ( y_1 \wedge y_2 ) = y_1 \wedge y_2
$. Similarly, $ y_1 = y_2 \wedge y_1 $.
\end{proof}

\subsection{Noise-type completion}
\label{2b}

Still, $ B \subset \La $ is a noise-type Boolean algebra. We define
\begin{equation}\label{2b1}
C = \{ x \in \Cl(B) : \exists y \in \Cl(B) \; ( x \wedge y = 0, \, x
\vee y = 1 ) \} \, .
\end{equation}

For $ x \in C $ we denote such $ y $ (unique by \ref{2a6}) by $ x'
$. We have
\begin{equation}\label{2b2}
B \subset C \subset \Cl(B) \, ,
\end{equation}
and for every $ x \in C $,
\begin{gather}
x' \in C \, ; \quad (x')' = x \, ; \label{2b3} \\
x \wedge x' = 0 \, , \quad  x \vee x' = 1 \, . \label{2b4}
\end{gather}
By \eqref{2a5a}, $ x, x' $ are independent; and by \ref{2a5},
\begin{equation}\label{2b5}
\forall x \in C \quad \Cl(B) \subset \La_{x,x'} \, .
\end{equation}

\begin{lemma}\label{2b6}
For every $ x \in C $ the map
\[
\Cl(B) \ni z \mapsto x \vee z \in \La
\]
is continuous.
\end{lemma}

\begin{proof}
Let $ z_n,z \in \Cl(B) $, $ z_n \to z $; we have to prove that $ x
\vee z_n \to x \vee z $. By \eqref{2a6a}, $ x' \wedge z_n \to x'
\wedge z $. Applying \ref{1d5} to $ ( x, x' \wedge z_n ) \in \La_x
\times \La_{x'} $ we get $ x \vee ( x' \wedge z_n ) \to x \vee ( x'
\wedge z ) $. It remains to prove that $ x \vee ( x' \wedge z_n ) = x
\vee z_n $ and $ x \vee ( x' \wedge z ) = x \vee z $. We prove the
latter; the former is similar. We have $ z \in \Cl(B) \subset
\La_{x,x'} $ by \eqref{2b5}. According to \ref{1d14} we apply the
lattice homomorphisms $ \La_{x,x'} \ni y \mapsto y \wedge x \in \La_x
$ and $ \La_{x,x'} \ni y \mapsto y \wedge x' \in \La_{x'} $ to $ y = x
\vee z $: $ ( x \vee z ) \wedge x = x $ and $ ( x \vee z ) \wedge x' =
( x \wedge x' ) \vee ( z \wedge x' ) = z \wedge x' $, therefore $ x
\vee z = x \vee ( x' \wedge z ) $.
\end{proof}

\begin{lemma}\label{2b7}
\[
\forall x \in C \;\, \forall y \in \Cl(B) \quad x \vee y \in \Cl(B) \,
.
\]
\end{lemma}

\begin{proof}
By \ref{2b6} it is sufficient to consider $ y \in B $. Applying
\ref{2b6} (again) to $ y \in B \subset C $ we see that the map $
\Cl(B) \ni z \mapsto y \vee z \in \La $ is continuous. This map sends
$ B $ into $ B $, therefore it sends $ x \in C \subset \Cl(B) $ into $
\Cl(B) $.
\end{proof}

\begin{lemma}\label{2b8}
For all $ x,y \in C $,
\[
x \vee y \in C \quad \text{and} \quad ( x \vee y )' = x' \wedge y' \,
.
\]
\end{lemma}

\begin{proof}
By \ref{2b7}, $ x \vee y \in \Cl(B) $. By \eqref{2a4a}, $ x' \wedge y'
\in \Cl(B) $. We have to prove that $ ( x \vee y ) \wedge ( x' \wedge
y' ) = 0 $ and $ ( x \vee y ) \vee ( x' \wedge y' ) = 1 $. We do it
using \ref{1d14} (similarly to the proof of \ref{2b6}).

First, $ x,y,x',y' \in C \subset \Cl(B) \subset \La_{x,x'} $.

Second, we consider $ z = ( x \vee y ) \wedge ( x' \wedge y' ) $ and
get $ z \wedge x = ( x \vee ( y \wedge x ) ) \wedge ( 0 \wedge ( y'
\wedge x ) ) = 0 $. Similarly, $ z \wedge x' = ( 0 \vee ( y \wedge x'
) ) \wedge x' \wedge ( y' \wedge x' ) \le y \wedge y'= 0 $. Therefore
$ z = 0 $, that is, $ ( x \vee y ) \wedge ( x' \wedge y' ) = 0 $.

Third, we consider $ z = ( x \vee y ) \vee ( x' \wedge y' ) $ and get
$ z \wedge x = x \vee ( y \wedge x ) \vee ( x' \wedge y' \wedge x ) =
x $. Similarly, $ z \wedge x' = ( x \wedge x' ) \vee ( y \wedge x' )
\vee ( x' \wedge y' \wedge x' ) = ( y \wedge x' ) \vee ( y' \wedge x'
) = ( y \vee y' ) \wedge x' = x' $. Therefore $ z = x \vee x' = 1 $,
that is, $ ( x \vee y ) \vee ( x' \wedge y' ) = 1 $.
\end{proof}

In addition, $ (x \vee y)' = x' \wedge y' \in C $; applying it to $
x', y' $ we see that $ x \wedge y \in C $ for all $ x,y \in C $, and
so,
\begin{equation}
C \text{ is a sublattice of } \La \, .
\end{equation}
The lattice $ C $ is distributive by \eqref{1d16}, since $ C \subset
\Cl(B) \subset \La_{x,x'} $ for all $ x \in C $. Also, $ 0 \in C $, $
1 \in C $, and each $ x \in C $ has a complement $ x' $ in $ C $. It
means that $ C $ is a Boolean lattice, that is, a Boolean algebra.

\begin{proposition}
The noise-type completion of $ B $ (as defined by \ref{1d6}) exists
and is equal to $ C $.
\end{proposition}

\begin{proof}
Being a Boolean algebra satisfying \eqref{2a5a}, $ C $ is a noise-type
Boolean algebra, and $ B \subset C \subset \Cl(B) $. We have to prove
that $ C $ contains any other noise-type Boolean algebra $ C_1 $ such
that $ B \subset C_1 \subset \Cl(B) $. This is easy: each $ x \in C_1
$ has a complement $ x' $ in $ C_1 \subset \Cl(B) $, therefore $ x \in
C $ just by \eqref{2b1}.
\end{proof}

Theorem 1 follows immediately. Also Theorem 2 follows,
since its conditions are mirrored by the definition of $ C $.

\bigskip
\filbreak
{
\small
\begin{sc}
\parindent=0pt\baselineskip=12pt
\parbox{4in}{
Boris Tsirelson\\
School of Mathematics\\
Tel Aviv University\\
Tel Aviv 69978, Israel
\smallskip
\par\quad\href{mailto:tsirel@post.tau.ac.il}{\tt
 mailto:tsirel@post.tau.ac.il}
\par\quad\href{http://www.tau.ac.il/~tsirel/}{\tt
 http://www.tau.ac.il/\textasciitilde tsirel/}
}

\end{sc}
}
\filbreak

\end{document}